\documentclass{article}
\usepackage{amssymb,amsmath,amsthm,bbm,tikz,mathrsfs,calc}
\usepackage[margin=2cm]{geometry}

\newtheorem{lemma}{Lemma}
\newtheorem{theorem}{Theorem}
\newtheorem{corollary}{Corollary}

\newtheorem{claim}{Claim}

\newcommand{\BF}[1]{{\bf\boldmath #1\unboldmath}}

\newcommand{\B}{\{0,1\}}

\newcommand{\feedback}{\tau}

\title{On fixable families of Boolean networks}

\author{
Maximilien Gadouleau \footnote{Department of Computer Science, Durham University, UK} \and
Adrien Richard \footnote{Laboratoire I3S, UMR CNRS 7271 \& Universit\'e C\^ote d'Azur, France}
}

\begin{document}

\maketitle

\begin{abstract}
The asynchronous dynamics associated with a Boolean network $f:\B^n\to\B^n$ is a finite deterministic automaton considered in many applications. The set of states is $\B^n$, the alphabet is $[n]$, and the action of letter $i$ on a state $x$ consists in either switching the $i$th component if $f_i(x)\neq x_i$ or doing nothing otherwise. This action is extended to words in the natural way. We then say that a word $w$ {\em fixes} $f$ if, for all states $x$,  the result of the action of $w$ on $x$ is a fixed point of $f$. A whole family of networks is fixable if its members are all fixed by the same word, and the fixing length of the family is the minimum length of such a word. In this paper, we are interested in families of Boolean networks with relatively small fixing lengths. Firstly, we prove that fixing length of the family of networks with acyclic asynchronous graphs is $\Theta(n 2^n)$. Secondly, it is known that the fixing length of the whole family of monotone networks is $O(n^3)$. We then exhibit two families of monotone networks with fixing length $\Theta(n)$ and $\Theta(n^2)$ respectively, namely monotone networks with tree interaction graphs and conjunctive networks with symmetric interaction graphs. 
\end{abstract}

\section{Introduction}


A \BF{Boolean network} (\BF{network} for short) is a finite dynamical system usually defined by a function  
\[
f:\B^n\to\B^n,\qquad x=(x_1,\dots,x_n)\mapsto f(x)=(f_1(x),\dots,f_n(x)).
\]
Boolean networks have many applications. In particular, since the seminal papers of McCulloch and Pitts \cite{MP43}, Hopfield \cite{H82}, Kauffman \cite{K69,K93} and Thomas \cite{T73,TA90}, they are omnipresent in the modeling of neural and gene networks (see \cite{B08,N15} for reviews). They are also essential tools in computer science, for network coding solvability \cite{GRF16} and memoryless computation \cite{BGT14,CFG14a}. 

The ``network'' terminology comes from the fact that the \BF{interaction graph} of $f$ is often considered as the main parameter of $f$: it is the directed graph $G(f)$ with vertex set $[n]:=\{1,\dots,n\}$ and an arc from $j$ to $i$ if $f_i$ depends on $x_j$, that is, if there exist $x,y\in\B^n$ that only differ in the component $j$ such that $f_i(x)\neq f_i(y)$.  

In many applications, for modelling gene networks in particular, the dynamics derived from $f$ is the asynchronous dynamics \cite{A-J16}. That is usually represented by the directed graph $\Gamma(f)$, called \BF{asynchronous graph} of $f$ and defined as follows. The vertex set of $\Gamma(f)$ is $\B^n$, the set of all the possible \BF{states}, and there is an arc from $x$ to $y$ if and only if $x$ and $y$ differs in exactly one component, say $i$, and $f_i(x)\neq x_i$. An example of a network with its interaction graph and its asynchronous is given in Figure~\ref{fig:f}. 

\begin{figure}
\[
\begin{array}{ccc}
\begin{array}{c}
\begin{array}{c|c}
	x & f(x)\\\hline
	000 & 000\\
	001 & 000\\
	010 & 001\\
	011 & 001\\
	100 & 010\\
	101 & 000\\
	110 & 010\\
	111 & 100\\
\end{array}
\quad
\begin{array}{l}
	f_1(x)= x_1 \land x_2 \land x_3\\
	f_2(x)= x_1 \land \neg x_3\\
	f_3(x)= x_2 \land \neg x_1.
\end{array}
\end{array}
&
\begin{array}{c}
\begin{tikzpicture}
\def\d{2}
\node (000) at (0,0){000};
\node (001) at (1,1){001};
\node (010) at (0,2){010};
\node (011) at (1,3){011};
\node (100) at (2,0){100};
\node (101) at (3,1){101};
\node (110) at (2,2){110};
\node (111) at (3,3){111};
\path[thick,->,draw,black]
(001) edge (000)
(010) edge (000)
(010) edge (011)
(011) edge (001)
(100) edge (000)
(100) edge (110)
(101) edge (001)
(101) edge (100)
(110) edge (010)
(111) edge (110)
(111) edge (101)
;
\end{tikzpicture}
\end{array}
&
\begin{array}{c}
\begin{tikzpicture}
\node[outer sep=1,inner sep=2,circle,draw,thick] (1) at (150:1){$1$};
\node[outer sep=1,inner sep=2,circle,draw,thick] (2) at (30:1){$2$};
\node[outer sep=1,inner sep=2,circle,draw,thick] (3) at (270:1){$3$};
\draw[->,thick] (1.128) .. controls (140:1.8) and (160:1.8) .. (1.172);
\path[->,thick]
(1) edge[bend left=15] (2)
(2) edge[bend left=15] (1)
(2) edge[bend left=15] (3)
(3) edge[bend left=15] (2)
(3) edge[bend left=15] (1)
(1) edge[bend left=15] (3)
;
\end{tikzpicture}
\end{array}
\\[25mm]
\text{$3$-component network $f$}&
\Gamma(f)&
G(f)
\end{array}
\]
\caption{\label{fig:f}A network $f$ given under two different forms (its table and a definition of its components by logical formulas) with its asynchronous graph $\Gamma(f)$ and its interaction graph $G(f)$.}
\end{figure}
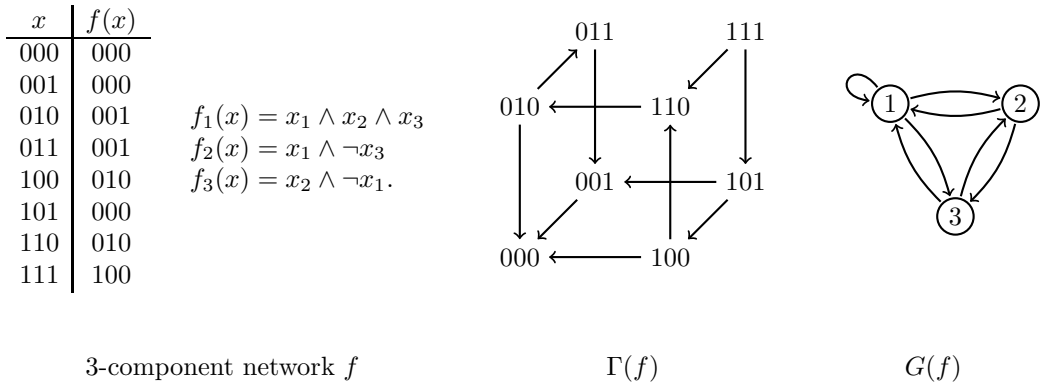

For any $i \in [n]$, the result of the action of the letter $i$ on $x$ is given by 
\[
f^i(x):=(x_1,\dots,f_i(x),\dots,x_n).
\] 
This action is extended to words on the alphabet $[n]$ in the natural way: the result of the action of a word $w=i_1i_2\dots i_k$ on a state $x$ is defined by 
\[
f^w(x):=(f^{i_k}\circ f^{i_{k-1}}\circ\dots\circ f^{i_1})(x).
\]

A word $w$ \BF{fixes} $f$ if $f^w(x)$ is a fixed point of $f$ for every $x$ \cite{AGRS18}. If $f$ admits a fixing word we say that $f$ is \BF{fixable}. For instance, the network in Figure~\ref{fig:f} is fixed by $w = 1231$, and is hence fixable. It is rather easy to see that $f$ is fixable if and only if there is a path in $\Gamma(f)$ from any initial state to a fixed point of $f$. We may think that the fixability is a strong property. However, somewhat surprisingly, for $n$ sufficiently large, more than half of the $n$-component networks are fixable:

\begin{theorem}[Bollob\'as, Gotsman and Shamir \cite{BGS93}]
Let $\phi(n)$ be the fraction of $n$-component networks that are fixable. Then 
\[
\lim_{n\to\infty} \phi(n)=1 - \frac{1}{e}.
\]
\end{theorem}

For any fixable network $f$, the \BF{fixing length} of $f$ is the minimum length of a word fixing $f$ and is denoted as $\lambda(f)$. For instance, we have seen that $w=1231$ fixes the network in Figure~\ref{fig:f}, thus this network has fixing length at most $4$, and it is easy to see that no word of length three fixes this network, and thus it has fixing length exactly $4$. It is easy to construct a fixable $n$-component network $f$ such that $\lambda(f)$ is exponential in $n$ (see \cite{AGRS18}).

We extend our concepts to entire families $\mathcal{F}$ of $n$-component networks. We say that $\mathcal{F}$ is fixable if there is a word $w$ such that $w$ fixes $f$ for all $f\in\mathcal{F}$, which is clearly equivalent to: all the members of $\mathcal{F}$ are fixable. The fixing length $\lambda(\mathcal{F})$ is defined naturally as the minimum length of a word fixing $\mathcal{F}$; we then have the upper bound $\lambda(\mathcal{F}) \le 4^n |\mathcal{F}|$. Indeed, let $\{0,1\}^n = \{x^1, \dots, x^{2^n}\}$ and $f \in \mathcal{F}$ and recursively define the word $W^f := w^1, \dots, w^{2^n}$ such that $f^{w^1, \dots, w^k}(x^k)$ is a fixed point for all $1 \le k \le 2^n$. Since every $w^k$ can be chosen to be of length at most $2^n$, we have $|W^f| \le 4^n$. Concatenating all the $W^f$ words then yields a word fixing $\mathcal{F}$ of length $4^n |\mathcal{F}|$. 

We are then interested in families $\mathcal{F}$ which can be fixed ``rapidly'', i.e. far below the trivial upper bound above. Firstly, we prove that the fixing length of the family of $n$-component networks with acyclic asynchronous graphs is in $\Theta(n 2^n)$. Secondly, it is known that the fixing length of family of $n$-component monotone networks is in $O(n^3)$ \cite{AGRS18}. We then exhibit two families of monotone networks which linear and quadratic fixing lengths: monotone networks with tree interaction graphs and conjunctive networks with symmetric interaction graphs. 

\section{Notation}

Let $w=w_1\dots w_p$ be a word. Then length $p$ of $w$ is denoted $|w|$. We say that a word $u=u_1\dots u_q$ is a \BF{subword} of $w$ is there exits $1\leq i_1 <i_2 < \dots < i_q\leq p$ such that $u=w_{i_1}\dots w_{i_q}$. The empty word is denoted $\epsilon$. A word $w$ is \BF{$n$-universal} if every permutation of $[n]$ (word of length $n$ without repetition) is a subword of $w$. The minimum length of an $n$-universal word is denoted $\lambda(n)$. We then have (see \cite{AGRS18} and references therein):
\[
\lambda(n) = n^2 - o(n^2).
\]

Directed graphs have no parallel arcs, and may have loops (arcs from a vertex to itself). Paths and cycles are  always without repeated vertices.  Given a directed graph $G$, the \BF{underlying undirected} graph $H$ of $G$ has the same vertex set, and two vertices $i$ and $j$ are adjacent in $H$ if and only if $i\neq j$ and $G$ has an arc from $i$ to $j$ or from $j$ to $i$. We refer the reader to the authoritative book on graphs by Bang-Jensen and Gutin \cite{BG08} for some basic concepts, notation and terminology.

Given $x,y\in\B^n$, we write $x\leq y$ to mean that $x_i\leq y_i$ for all $i\in [n]$. Equipped with this partial order, $\B^n$ is the usual Boolean lattice. An $n$-component network $f$ is \BF{monotone} if it preserves this partial order, that is,
\[
\forall x,y\in\B^n,\qquad x\leq y\Rightarrow f(x)\leq f(y).
\]
We denote by $F_M(n)$ the family of $n$-component monotone networks and by $\lambda_M(n)$ the fixing length of $F_M(n)$. More generally, if $F_X(n)$ is any family of $n$-component fixable networks, then $\lambda_X(n)$ is the fixing length of $F_X(n)$. If $G$ is a directed graph, then $F(G)$ denotes the set of $n$-component networks $f$ such that the interaction graph of $f$ isomorphic to a subgraph of $G$. Then, $F_X(G):=F_X(n)\cap F(G)$ and $\lambda_X(G)$ is the fixing length of $F_X(G)$.

Let $f$ be an $n$-component network. We set $f^\epsilon:= \mathrm{id}$ and, for any integer $i$ and $x\in\B^n$, we define $f^i(x)$ as in the introduction if $i\in [n]$, and $f^i(x):=x$ if $i\not\in [n]$. This extends the action of letters in $[n]$ to letters in $\mathbb{N}$, and by extension, this also defines the action of a word over the alphabet $\mathbb{N}$. 

Let $G$ be a directed graph with vertex set $[n]$. The \BF{conjunctive network on $G$} is the $n$-component network $f$ defined as follows: for all $i\in[n]$ and $x\in\B^n$, 
\[
f_i(x)=\bigwedge_{j\in N^-(i)} x_j,
\]
where $N^-(i)$ is the set of in-neighbors of $i$ in $G$ (with $f_i(x)=1$ if $i$ has no in-neighbor).  As shown in \cite{AGRS18}, every conjunctive network on $n$ vertices has a fixing length of at most $2n-2$, and this is tight. 

\section{Asynchronous-acyclic networks}

A network is \BF{asynchronous-acyclic} if its asynchronous graph is acyclic. The family of all such networks is denoted as $F_{A\Gamma}(n)$. Clearly, any asynchronous-acyclic network is fixable. It is well known, and easy to prove, that if $G(f)$ is acyclic then so is $\Gamma(f)$ \cite{Rob80}. 

Let $P=x^1x^2\dots x^l$ be a path of the $n$-cube, and let $i_k$ be the component that differs between $x^k$ and $x^{k+1}$, $1\leq k<l$. The word $i_1i_2\dots i_{l-1}$ is the word \BF{induced} by $P$, and a word is an \BF{$n$-path-word} if it is induced by at least one path of the $n$-cube. A word $w$ is an $n$-path-word if and only if, for all $i < j$, there exists $k \in [n]$ which occurs an odd number of times in $w_i \dots w_j$ \cite{Gil58}. Note that an $n$-path-word has no consecutive repetitions and is of length at most $2^n-1$. A word $W$ is \BF{$n$-path-universal} if it contains all $n$-path-words as subwords. Let $\Lambda(n)$ denote the minimum length of an $n$-path-universal word. For instance, for $n=2$, the maximal $n$-path-words are $121$ and $212$, hence the word $1212$ is $n$-path-universal and $\Lambda(2) = 4$. In general, $\Lambda(n)=\Theta(n2^n)$:

\begin{lemma}
For all $n\geq 1$,
\[
	 n 2^{n-1} \le \Lambda(n) \le (n-1) (2^n - 1) + 1.
\]
\end{lemma}

\begin{proof}
We first show the lower bound. For that, we define inductively a Hamiltonian path $P^n$ of the $n$-cube in the following way: $P^1:=x^1x^2$ with $x^1=0$ and $x^2=1$ and, for $n>1$, $P^n$ is defined from $P^{n-1}=x^1x^2\dots x^{2^{n-1}}$ by setting 
\[
P^n:=(0,x^1)(0,x^2)\dots (0,x^{2^{n-1}})(1,x^{2^{n-1}})(1,x^{2^{n-1}-1})\dots (1,x^2)(1,x^1).
\] 
$P^n$ then corresponds to the canonical Gray code. Let $w^n$ be the word induced by $P^n$. It is easy to see that the letter $n$ appears exactly $2^{n-1}$ times in $w^n$. Thus any $n$-path-universal word contains at least $2^{n-1}$ occurrences of the letter $n$. By symmetry, every letter appears at least $2^{n-1}$ times, thus $\Lambda(n) \ge n 2^{n-1}$. 

We now show the upper bound. Let 
\[
W:=1,u^1,u^2,\dots,u^{2^n-1}\quad\text{with}\quad
\left\{
\begin{array}{ll}
u^k:=2,3,\dots, n&\text{ if $k$ is odd}\\
u^k:=n-1,n-2,\dots,1&\text{ if $k$ is even}.
\end{array}
\right.
\]
Then $W$ contains all words of length at most $2^n-1$ without consecutive repetitions, and hence is $n$-path-universal.
\end{proof}

\begin{theorem} \label{prop:h_acyclic}
A word fixes $F_{A\Gamma}(n)$ if and only if it is $n$-path-universal, thus 
\[
\lambda_{A\Gamma}(n) = \Lambda(n).
\]
Moreover, 
\[
\max_{f \in F_{A\Gamma}(n)} \lambda(f) = 2^n - 1.
\]
\end{theorem}

\begin{proof}
Let $P$ be any path of the $n$-cube, and let $w$ be the word induced by $P$. Consider the $n$-component network $f$ whose asynchronous graph only has the arcs contained in $P$. Then $f$ is clearly asynchronous-acyclic and a word fixes $f$ if and only if it contains $w$ as a subword. We deduce the following three properties (the third is obtained from the second using the fact that the $n$-cube has a Hamiltonian path, e.g. the path $P^n$ constructed above):
\begin{enumerate}
\item[{\it (1)}] Any word fixing $F_{A\Gamma}(n)$ is $n$-path-universal, thus $\lambda_{A\Gamma}(n)\geq\Lambda(n)$.
\item[{\it (2)}] For any $n$-path-word $w$, there exists $f\in F_{A\Gamma}(n)$ with $\lambda(f)\geq |w|$.
\item[{\it (3)}] There exists $f\in F_{A\Gamma}(n)$ with $\lambda(f)\geq 2^n-1$.
\end{enumerate}

Conversely, let us prove that $\lambda(f) \le 2^n - r$ for any $f \in F_{A\Gamma}(n)$ with $r$ fixed points. Since $r\geq 1$, together with the property {\it (3)}, this shows the second assertion of the statement. Let  $x^1x^2\dots x^{2^n}$ be a topological sort of $\Gamma(f)$: for all $1\leq p\leq q\leq 2^n$, $\Gamma(f)$ has no arc from $x^q$ to $x^p$. Then $x^p$ is a fixed point if and only if $p>2^n-r$. Furthermore, we have the following two properties: {\it (i)} for all $i\in [n]$ we have $f^i(x^p)=x^q$ for some $q\geq p$; and {\it (ii)}
if $p\leq 2^n-r$ then there exists at least one component in $[n]$, say $i_p$, such that $f^{i_p}(x^p)=x^q$ for some $q>p$. We will prove that $w:=i_1i_2\dots i_{2^n-r}$ fixes $f$. Let $1\leq p\leq 2^n$, and let us prove, by induction on $k$, that:
\[
\forall 1\leq k\leq 2^n-r,\qquad f^{i_1i_2\dots i_k}(x^p)=x^q\text{ for some $q>k$}. 
\]
If $k=1$, we deduce that $f^{i_1}(x^p)=x^q$ for some $q>1$ from {\it (i)} if $p\geq 2$, and from {\it (ii)} if $p=1$. Suppose now that $k>1$. By induction, $f^{i_1i_2\dots i_{k-1}}(x^p)=x^l$ for some $l>k-1$. We then deduce that $f^{i_1i_2\dots i_k}(x^p)=f^{i_k}(x^l)=x^q$ for some $q>k$ from {\it (i)} if $l>k$, and from {\it (ii)} if $l=k$. This completes the induction. The particular case $k=2^n-r$ shows that $f^w(x^p)=x^q$ for some $q>2^n-r$, and thus $f^w(x^p)$ is a fixed point. 

It remains to prove that any $n$-path-universal word $W=j_1j_2\dots j_s$ fixes $f$. Let $y^1\in\B^n$, and for all $1\leq k\leq s$, let 
\[
y^{k+1}:=f^{j_1j_2\dots j_k}(y^1)\qquad\text{(or, equivalently, $y^{k+1}:=f^{j_k}(y^k)$)}. 
\]
Let us prove that $y^{s+1}=f^W(y^1)$ is a fixed point. This is clear if $y^1$ is a fixed point. Otherwise $y^1\neq y^{s+1}$. Then, in the sequence $y^1y^2\dots y^{s+1}$, let $a_1,\dots a_t$ be the positions such that $y^{a_k}\neq y^{a_k+1}$. The states visited by the sequence then correspond to the path $P:=y^{a_1}y^{a_2}\dots y^{a_t}y^{a_{t+1}}$ of $\Gamma(f)$, where $a_{t+1}:=s+1$, and $w:=j_{a_1}j_{a_2}\dots j_{a_t}$ is the word induced by this path. Suppose, for a contradiction, that $y^{a_{t+1}}$ is not a fixed point. Let $y^{a_{t+2}}$ be an out-neighbor $y^{a_{t+1}}$ in $\Gamma(f)$, and let $i$ the component that differs between these two states. Then $y^{a_1}y^{a_2}\dots y^{a_{t+1}}y^{a_{t+2}}$ is a path of $\Gamma(f)$, and $w':=j_{a_1}j_{a_2}\dots j_{a_t}i$ is the word induced by this path. Hence, $w'$ is a subword of $W$. By construction, $a_1a_2\dots a_t$ corresponds to the first occurrence of $w$ in $W$. That is, setting $a_0:=0$, we have the following: for all $1\leq k\leq t$, $a_k$ is the first position in $W$ greater than $a_{k-1}$ where the letter $j_{a_k}$ appears. Since $w'$ is a subword of $W$, we deduce that $i$ appears after the position $a_t$, that is, there exists $a_t<b\leq s$ such that $j_b=i$. By the definition of the sequence $a_1,\dots a_t$, we have $y^k=y^{s+1}$ for all $a_t<k\leq s+1$, and thus $y^b=y^{s+1}=y^{a_{t+1}}$. But then, $y^{b+1}=f^i(y^b)=f^i(y^{a_{t+1}})=y^{a_{t+2}}\neq y^{a_{t+1}}=y^b$, a contradiction. Thus every $n$-path-universal word fixes $F_{A\Gamma}(n)$. Thus $\lambda_{A\Gamma}(n)\leq\Lambda(n)$, and with the property {\it (1)} we obtain an equality. 
\end{proof}

\section{Monotone networks}

It is proved in \cite{AGRS18} that the fixing length of the family of $n$-component monotone networks $F_M(n)$ is $O(n^3)$ and $\Omega(n^2)$; closing the gap is challenging. In this section, we introduce two classes of monotone networks with linear and quadratic fixing lengths respectively. The first family is based on trees. We say a directed graph is \BF{loop-full} if every vertex has a loop.

\begin{theorem} \label{thm:tree}
Let $G$ be a loop-full tree with $n$ vertices and $L$ leaves. Then
\[
	\lambda_M(G) = 2n - L - 1.
\]
\end{theorem}

\begin{proof}
The result is clear for $n \le 2$, thus we assume $n \ge 3$ henceforth. Then $G$ has a non-leaf, say $r$; we then root $G$ at $r$. We order the non-leaf vertices of in non-decreasing order of distance from the root (and hence $r=1$ and $1, \dots, N := n-L$ are non-leaves) and we denote the leaves as $N+1, \dots, n$ (in no particular order). Note that according to this order, $i \le j$ if and only if the path from $j$ to $r$ goes through $i$.

\medskip

We first prove that $\lambda_M(G) \le 2n-L-1 = 2N - 1 + L$. Let $W^1 := 1$ and $W^i := i W^{i-1} i$ for all $i \in [N]$; therefore, 
\[
	W := W^N = N, N-1, \dots, 2, 1, 2, \dots N
\]
has length $2N-1$. Let $f \in F_M(G)$ and $x \in \{0,1\}^n$. We say that $x$ is fixed on $[i]$ if $f_j(x) = x_j$ for all $j \in [i]$.

\begin{claim}
For all $i \in [N]$, $f^{W^i}(x)$ is fixed on $[i]$.
\end{claim}
\begin{proof}
This is obvious for $i=1$, thus assume that $i > 1$. Let
\[
x^1:=f^i(x),\quad x^2:=f^{W^{i-1}}(x^1),\quad x^3:=f^i(x^2).
\]
By induction, $x^2$ is fixed on $[i-1]$, and we want to prove that $f^{W^i}(x)=x^3$ is fixed on $[i]$. Since we have have $f_i(x^3)=x^3_i$, if $x^3$ is not fixed on $[i]$ there exists $j\in [i-1]$ such that $f_j(x^3)\neq x^3_j=x^2_j=f_j(x^2)$. Thus $x^3\neq x^2$, and this implies $x^3_i\neq x^2_i$, and since $i$ is the only component that differs between these two states, there is an arc from $i$ to $j$ in $G$. Assume that $x^3_i>x^2_i$, the other case being similar. Then $f_i(x^2)=x^3_i=1$, and $x^3\geq x^2$. Since $f_j$ is monotone, we deduce that $f_j(x^3)>f_j(x^2)=x^2_j=0$. Thus $x^2_i=x^2_j=0$. Since, $i$ is a leaf in the subgraph of $G$ induced by $[i]$, and since $j$ is adjacent to $i$ in $G$, we deduce that, in $G$, $i$ has no in-neighbors in $[i]\setminus \{i,j\}$. Since $x^2_i=x^2_j=0$, we deduce that, for all in-neighbors $k$ of $i$ in $G$, we have $x_k\geq x^2_k$, with an equality if $k\neq i,j$ (since then $k$ does not appear in $W^i$). Since $f_i$ is monotone, we deduce that $f_i(x)\geq f_i(x^2)=1$. Thus $f_i(x)=1$ and we deduce that $x^1_i=1$. Since $x^1_i=x^2_i$ (because $i$ does not appear in  $W^{i-1}$), we obtain a contradiction. Thus $f^{W^i}(x)$ is fixed on $[i]$ for all $i\in [N]$. 
\end{proof}

In particular, $y := f^W(x)$ is fixed on $[N]$. Let $\Omega = W, N+1, \dots, n$ and $z := f^{N+1, \dots, n}(y) = f^\Omega(x)$. Then we claim that $z$ is a fixed point of $f$. First, it is easy to check that $f_l(z) = z_l$ for any leaf $l$. Second, by the claim above, $f_p(z) = f_p(y) = y_p = z_p$ for any non-leaf $p$ which is not adjacent to any leaf. All that is left to show is that non-leaves which are adjacent to some leaves are still fixed by $f^\Omega$. Let $m \in N$ be a non-leaf, $\Lambda$ be the set of leaves adjacent to $m$ and $P$ be the other neighbours of $m$. Then $z_P = y_P$, $z_m = y_m = f_m( y_P, y_\Lambda )$ (by the claim) and $z_l = f_l ( y_m, y_l )$ for all $l \in \Lambda$. Suppose that $y_m = 0$ (the case $y_m = 1$ is similar). Then $z_l \le y_l$, since otherwise we have $z_l = 1 = f_l(0,0)$, which implies that $f_l$ is either constant or non-monotonic. Thus $f_m( z_P, z_\Lambda ) \le f_m ( y_P, y_\Lambda ) = y_m = z_m = 0$ and $m$ is indeed fixed.

\medskip

We now prove that $\lambda_M(G) \ge 2N-1 + L$. Let $f$ be the conjunctive network on $G$ and let $w$ be a word fixing $f$. Firstly, every $i \in [n]$ appears in $W$. Indeed, let $j \ne i$ and $x \in \B^n$ such that $x_j = 1$ and $x_k = 0$ for all $k \ne j$. Then the only fixed point reachable by $x$ is the all-zero state and hence the value of $x_i$ must be updated. This first claim is sufficient to prove the lower bound when $N = 1$, thus we assume $N \ge 2$ in the sequel. Secondly, $W$ contains every sequence of the form $ij$ as a subword for any distinct $i,j \in [N]$. Indeed, let $u, i, \dots, j$ be a path in $G$. By considering the state $x \in \B^n$ such that $x_u = 1$ and $x_k = 0$ for all $k \ne u$, we see that $W$ must contain $ij$ as a subword. Now consider the subword $W'$ of $W$ only containing the occurrences of letters from $N$. Let $k \in [N]$ be the letter whose first occurrence in $W'$ happens last (thus the first occurrence of $k$ is in position $q \ge N$). Since $W'$ contains $ki$ as a subword for all $i \in [N]$, we see that $|W'| \ge N + (N-1)$. By the first claim, $|W| \ge |W'| + L \ge 2N-1 + L$.
\end{proof}

The \BF{circumference} of a directed graph $G$ is the length of a longest cycle in $G$, and zero if $G$ is acyclic. It is easy to verify that $G$ has circumference at most two if and only if the underlying undirected graph of each strong component of $G$ is a trees. An \BF{$l$-feedback vertex set} is a set of vertices $I$ such that $G \setminus I$ has circumference at most $l$. The \BF{$l$-feedback number} of $G$, denoted as $\feedback_l(G)$, is the minimum cardinality of an $l$-feedback vertex set of $G$. In \cite{AGRS18}, it is shown that, for every directed graph $G$ with $n$ vertices,
\[
\lambda_M(G)\leq (2\feedback_1(G)^2 + 1) n
\]
We now show that a bounded $2$-transversal number in $G$ implies that $F_M(G)$ has quadratic fixing length.

\begin{corollary}
For every directed graph $G$ with $n$ vertices, 
\[
\lambda_M(G)  \le \feedback_2(G) n^2+3n.
\]
\end{corollary}

\begin{proof}
Let $\tau := \feedback_2(G)$, $\alpha := n - \tau$, and let us label the vertices of $G$ from $1$ to $n$ in such a way that $I = \{\alpha + 1, \dots, n\}$ is a minimum $2$-feedback vertex set of $G$. Let $T_1, \dots, T_k$ be the strong components of $G \setminus I$ in topological order. Note that $\alpha \ge 2$ and that the underlying undirected graph of every $T_i$ is a tree, say with $n_i$ vertices. Let $w^i$ be a word of length at most $2n_i-1$ fixing each $F_M(T_i)$, which exists in virtue of the previous theorem. It is then easy to see that the word $w = w^1, \dots, w^k$ fixes $F_M(G\setminus I)$, and $|w| \le 2\alpha-k\leq 2n$.

By adapting the proof of \cite[Theorem 13]{AGRS18}, we can prove the following result. If $w$ is a word fixing $F_M(G \setminus I)$ and if, for all $1 \le k \le \tau$, $\omega^k$ is an $(\alpha + k)$-universal word, then $F_M(G)$ is fixed by the word
\[
W := w, \alpha + 1, \omega^1,\alpha+2,\omega^2, \dots, n, \omega^\tau.
\]
Applying this result to our problem yields a word $W$ of length
\[
	 |W| = |w| + \tau+\sum_{k=1}^\tau \lambda(\alpha + k) \le 2n  + n+\sum_{i= \alpha +1}^n i^2 \le \tau n^2+3n.
\]
\end{proof}

The second family is described as follows. A directed graph is \BF{symmetric} if, for all distinct vertices $i$ and $j$, if there is an arc from $i$ to $j$ then there is also an arc from $j$ to $i$ (thus a symmetric directed graph can be regarded as an undirected graph with possibly a loop on some vertices). We consider the family $F_{CS}(n)$ of conjunctive networks on symmetric directed graphs.

Say a word $w$ over the alphabet $[n]$ is \BF{$(n,k)$-universal} if it contains, as subwords, all the words of length $n-k$ without repetition (there are $(n-k)!{n\choose k}$ such words). Let $\lambda_k(n)$ denote the minimum length of an $(n,k)$-universal word. In particular, $\lambda_0(n)=\lambda(n)$ and $\lambda_k(n)=0$ for $k\geq n$. 

\begin{lemma}
For all fixed $k$, $\lambda_k(n) = n^2 - o(n^2)$. 
\end{lemma}

\begin{proof}
Let $W^k:=1,w^1,w^2,\dots,w^{n-k}$ with $w^r:=2,3,\dots, n$ if $r$ is odd and $w^k:=n-1,n-2,\dots,1$ otherwise. Then it is easy to check that $W^k$ is $(n,k)$-universal. Thus
\[
\lambda_k(n) \le (n-1)(n-k) + 1.
\]
Furthermore, it $\omega^k$ and $\omega^{n-k}$ are any two words that are $(n,k)$-universal and $(n,n-k)$-universal respectively, then $\omega:=\omega^k,\omega^{n-k}$ is $n$-universal. This shows that $\lambda(n)\leq \lambda_k(n)+\lambda_{n-k}(n)$, and from the upper bound above we obtain 
\[
\lambda_k(n)\geq \lambda(n)-(n-1)k-1.
\]
Since $\lambda(n) = n^2 - o(n^2)$ as $n$ tends to infinity, this proves the lemma. 

\end{proof}

By this lemma and the theorem below, the fixing length of $F_{CS}(n)$ is $\Theta(n^2)$.

\begin{theorem}
For all $n\geq 1$,
\[
	\lambda_1(n) \le \lambda_{CS}(n) \le \lambda_2(n) + n.
\]
\end{theorem}

\begin{proof}
We first prove the lower bound. This is obvious for $n\geq 1$, so suppose that $n>1$. Let $i_1i_2\dots i_n$ be a permutation of $[n]$, and let $f$ be the $n$-component conjunctive network defined by $f_{i_1}(x)=x_{i_1}\land x_{i_2}$, $f_{i_n}(x)=x_{i_n}\land x_{i_{n-1}}$ and $f_{i_k}(x)=x_{i_{k-1}}\land x_{i_{k+1}}$ for all $1<k<n$. Then the interaction graph of $f$ is symmetric, and the arguments in the proof of Theorem~\ref{thm:tree} shows that any word fixing $f$ contains $i_2\dots i_n$ as subword (as well as $i_{n-1}i_{n-2}\dots i_1$). Hence, any word fixing $F_{CS}(n)$ contains all the words of length $n-1$ without repetition, and is thus of length at least $\lambda_1(n)$.

We now prove the upper bound. Consider the word $W := w, \omega$, where $w := 12 \dots n$ and $\omega$ is a shortest $(n,2)$-universal word. Then $|W|=\lambda(n)+n$ and we will prove that $W$ fixes $F_{CS}(n)$. Let $G$ by any symmetric directed graph with vertex set $[n]$, and let $f$ the conjunctive network on $G$. Let $H$ be the underlying undirected graph of $G$. Let $H_1\dots H_p$ be the connected components of $H$, with vertices $V_1,\dots ,V_p$. Let any $x\in\B^n$, $y := f^w(x)$, $z:=f^\omega(y)=f^W(x)$ and $1\leq q\leq p$. 

Suppose first $|V_q|=1$, say $V_q=\{i\}$. Then there are two possibilities: either $i$ has no loop in $G$, and then $f_i=1$ is constant, thus $y_i=z_i=f_i(z)=1$; or $i$ has a loop and then $x_i=y_i=z_i=f_i(z)$. Suppose now that $|V_q|\geq 2$. 

\begin{claim}
Either $y_i=1$ for all $i\in V_k$, or $y_i=y_j=0$ for some edge of $H_q$.
\end{claim}

\begin{proof}
For the sake of contradiction, suppose that the subgraph of $H_q$ induced by $\{i\in V_q : y_i = 0\}$ contains an isolated vertex, say $i$. Let $j$ be adjacent to $i$ in $H_k$. If $i<j$ then $f^{1 \dots j-1}(x)_i=y_i=0$ and thus $y_j = f^{1 \dots j}(x)_j=f_j(f^{1 \dots j-1}(x))=0$, a contradiction. Thus $k<i$ for every vertex $k$ adjacent to $i$ in $H_q$. Then $i$ has a loop, since otherwise $y_i = f_i(f^{1\dots i-1}(x))= \bigwedge_{k \in N^-(i)} y_k = 1$. We deduce that 
\[
y_i = x_i \land \bigwedge_{k \in N^-(i)\setminus\{i\}} y_k = 0,
\]
thus $x_i=0$, but then we would have $y_k = 0$ for any neighbour $k$ of $i$, a contradiction. This proves the claim. 
\end{proof}

If $y_i=1$ for all $i\in V_q$, then we clearly have $y_i=z_i=f_i(z)=1$ for all $i\in V_q$. Otherwise, $y_i=y_j=0$ for some edge $ij$ of $H_k$. It is easy to check that $f^u(y)_i=f^u(y)_j=0$ for any word $u$. Thus the states of $i$ and $j$ are always blocked in zeroes when performing the updates in $\omega$, which allows these two zeroes to be propagated to the whole component $H_q$. Since any vertex $k$ in $H_q$ is reachable from the edge $ij$ by a path of length at most $n-2$, it is easy to verify that $z_k=0$ for all $k\in V_q$, and we deduce that $f_k(z)=0=z_k$ for all $k\in V_q$. Hence, in any case, $f_k(z)=z_k$ for all $k\in V_q$, and thus $z$ is a fixed point of $f$ as desired. 
\end{proof}



\end{document}